\documentclass[a4paper,11pt]{amsart}
\usepackage{fouriernc}
\usepackage[T1]{fontenc}

\usepackage{graphicx}
\usepackage{amsmath}
\usepackage{amssymb}
\usepackage[active]{srcltx}
\usepackage{hyperref}
\usepackage{bbm}
\usepackage{enumerate}

\usepackage{geometry}
\newtheorem{thm}{Theorem}%
\newtheorem*{thm*}{Theorem}
\newtheorem{lem}{Lemma}%[section]%F
\newtheorem*{cor*}{Corollary}%[section]
\theoremstyle{definition}

\theoremstyle{remark}
\newtheorem{ex}{Example}%[section]

\theoremstyle{plain}

\def\NN{{\mathbb N}}

\def\RR{{\mathbb R}}
\def\TT{{\mathbb T}}

\def\ZZ{{\mathbb Z}}

\def\scrA{{\mathcal A}}
\def\scrB{{\mathcal B}}

\def\scrL{{\mathcal L}}
\def\scrM{{\mathcal M}}

\def\cl{\operatorname{cl}}

\def\dist{\operatorname{dist}}

\def\C{\operatorname{C{}}}

\def\GL{\operatorname{GL}}

\def\supp{\operatorname{supp}}

\def\vol{\operatorname{vol}}

\numberwithin{equation}{section}

\title{Pair correlation and equidistribution on manifolds}
\author{Jens Marklof}
\address{Jens Marklof, School of Mathematics, University of Bristol, Bristol BS8 1TW, U.K.\newline \rule[0ex]{0ex}{0ex} \hspace{8pt}{\tt j.marklof@bristol.ac.uk}}
\date{28 February 2019/3 April 2019}


\begin{document}

\begin{abstract}
This study is motivated by a series of recent papers that show that, if a given deterministic sequence in the unit interval has a Poisson pair correlation function, then the sequence is uniformly distributed. Analogous results have been proved for point sequences on higher-dimensional tori. The purpose of this paper is to describe a simple statistical argument that explains this observation and furthermore permits a generalisation to  bounded Euclidean domains as well as compact Riemannian manifolds. 
\end{abstract}

\maketitle

\section{Introduction}

A sequence of real numbers $\xi_1,\xi_2,\xi_3,\ldots$ in the unit interval $[0,1]$ is called {\em uniformly distributed} if, for any subinterval $[a,b]\subset[0,1]$, we have
\begin{equation}\label{ud-intro}
\lim_{N\to\infty} \frac{\#\big\{ j\leq N \mid \xi_j\in[a,b]\big\} }{N} = b-a .
\end{equation}
That is, the proportion of elements that fall into a given subinterval is asymptotic to its length. A classic example is the Kronecker sequence $\xi_j=\langle j\alpha \rangle$ (where $\langle \,\cdot\,\rangle$ denotes the fractional part), which is uniformly distributed if and only if $\alpha$ is irrational. %Weyl \cite{Weyl16} proved in 1916 that uniform distribution holds also for $\xi_j=\langle j^k\alpha \rangle$ for any fixed integer $k\geq 1$ and irrational $\alpha$. Establishing the equidistribution of deterministic sequences in a given topological space is still a major focus in many branches of mathematics, especially in number theory and ergodic theory. 
Once uniform distribution of a sequence is established, it is natural to investigate statistical properties on finer scales. One of the simplest such statistics is {\em pair correlation}. We say the sequence $(\xi_j)_{j\in\NN}$ in $[0,1]$ has a {\em Poisson pair correlation}, if for any bounded interval $[a,b]\subset\RR$ we have
\begin{equation}\label{pair-intro}
\lim_{N\to\infty} \frac{\#\big\{ (j_1,j_2)\in[1,N]^2 \mid \xi_{j_1}-\xi_{j_2} \in[\frac{a}{N},\frac{b}{N}],\; j_1\neq j_2\big\} }{N} = b-a .
\end{equation}
The average gap between the first $N$ elements $\xi_1,\ldots,\xi_N\in[0,1]$ is $\frac1N$, and so, by rescaling the interval to $[\frac{a}{N},\frac{b}{N}]$, we indeed measure correlations in units of the average gap size. The reference to {\em Poisson} stems from the fact that the right hand side of \eqref{pair-intro} corresponds to the pair correlation of a Poisson point process in $\RR$ of intensity one. What is more, the convergence \eqref{pair-intro} holds almost surely, if $(\xi_j)_{j\in\NN}$ is a sequence of independent, uniformly distributed random variables in $[0,1]$. 
Even for simple deterministic sequences, however, the convergence of pair correlation measures remains a significant challenge. 
For instance \eqref{pair-intro} is known to hold for $\xi_j=\langle j^k\alpha \rangle$ ($k\geq 2$ a fixed integer) for Lebesgue-almost every $\alpha$ \cite{Rudnick98}, and a lower bound on the Haussdorff dimension of permissible $\alpha$ has recently been established \cite{Aistleitner18b}. But so far there is not a single explicit example of $\alpha$, such as $\alpha=\sqrt 2$ or $\alpha=\pi$, for which \eqref{pair-intro} holds; not even in the quadratic case $k=2$ \cite{HeathBrown10,Marklof03,Marklof18}. There has been significant recent progess in characterising the Poisson pair correlation \eqref{pair-intro} for general sequences  $\xi_j=\langle a_j \alpha \rangle$, for Lebesgue-almost every $\alpha$, in terms of the additive energy of the integer coefficients $a_j$; cf.~\cite{Aistleitner18c,Bloom18} and references therein. Explicit examples for which Poisson pair correlation \eqref{pair-intro} can be established include the fractional part of square-roots, i.e., $\xi_j=\langle j^{1/2} \rangle$ \cite{ElBaz15}, and directions of points in a shifted Euclidean lattice \cite{ElBaz15b}. Note that \eqref{pair-intro} fails for the Kronecker sequence for any choice of $\alpha$ \cite{Larcher18,Marklof00}. Another interesting case of a uniformly distributed sequence is $\xi_j=\langle p_j \alpha \rangle$, where $p_j$ denotes the $j$th prime and $\alpha$ is irrational: also here \eqref{pair-intro} fails to hold, for almost every $\alpha$ \cite{Walker18}. This illustrates the perhaps unsurprising fact: uniform distribution does not imply Poisson pair correlation.

In two independent papers, Aistleitner, Lachmann and Pausinger \cite{Aistleitner18} and Grepstad and Larcher \cite{Grepstad17} reversed the question and asked whether Poisson pair correlation \eqref{pair-intro} of a given sequence implies uniform distribution. The answer is yes, even under weaker hypotheses than \eqref{pair-intro}, for sequences in the unit interval \cite{Aistleitner18,Grepstad17,Steinerberger18}. The same has been established for point sequences on higher-dimensional tori \cite{Hinrichs18,Steinerberger19}. In the present paper we develop a statistical argument that permits a generalisation of these findings to bounded domains in $\RR^d$ (Section \ref{sec1}) as well as compact Riemannian manifolds (Section \ref{sec2}; the special case of flat tori is discussed in the appendix). Instead of point sequences, we furthermore consider the more general setting of triangular arrays, i.e., sequences of finite point sets with increasing cardinality.

\section{Bounded domains}\label{sec1}

Let $\Omega\subset\RR^d$ be bounded with $\vol\partial\Omega=0$, where $\vol$ denotes the Lebesgue measure in $\RR^d$. (All subsets of $\RR^d$ in this paper are assumed to be Borel sets.)
Consider the triangular array $\xi=(\xi_{ij})_{ij}$ with coefficients $\xi_{ij}\in\Omega$ and indices $i,j\in\NN$, $j\leq N_i$, for some given $N_i\in\NN$ such that $N_i<N_{i+1}$. 

\begin{ex}\label{ex1}
Let $\Omega=[0,1]$. Take a real sequence $(\xi_j)_{j\in\NN}$ in $[0,1]$ (as in the introduction) and set $\xi_{ij}=\xi_j$ for $j\leq N_i=i\in\NN$. Sequences may thus be realised as special cases of a triangular arrays.
\end{ex}
\begin{ex}\label{ex2}
Let $\Omega=B_1^d$ be the open unit ball centered at the origin. Take a sequence $(a_j)_{j\in\NN}$ in $\RR^d$ such that $\|a_j\|\to\infty$, and set $\xi_{ij}=T_i^{-1} a_j$, with $N_i=\#\{ j\mid \|a_j\|<T_i\}$ and $T_1<T_2<\ldots\to\infty$ increasing sufficiently fast so that $N_{j+1}>N_j$.
\end{ex}

We associate with the $i$th row of $\xi$ the Borel probability measure $\nu_i$ on $\Omega$, defined by
\begin{equation}
\nu_i f = \frac{1}{N_i} \sum_{j=1}^{N_i} f(\xi_{ij}) ,
\end{equation}
where $f\in\C_b(\Omega)$ (bounded and continuous).
In other words, $\nu_i$ represents $N_i$ normalised point masses at the points $\xi_{i 1},\ldots,\xi_{i N_i}$.

Given a Borel probability measure $\sigma$ on $\Omega$, we say the triangular array $\xi$ is {\em equidistributed in $(\Omega,\sigma)$} if $\nu_i$ converges weakly to $\sigma$; that is,
\begin{equation}\label{dom-ud}
\lim_{i\to\infty} \nu_i f = \sigma f  \quad \text{for every $f\in\C_b(\Omega)$.}  
\end{equation}

In the case of Example \ref{ex1}, equidistribution in $([0,1),\vol)$ corresponds to the classical notion of uniform distribution discussed in the introduction. 

Let $A: \cl\Omega \to \GL(d,\RR)$ be a  continuous map. This means in particular that $\Delta(x)=|\det A(x)|$ is bounded above and below by positive constants. Define the finite Borel measure $\sigma$ on $\Omega$ by
\begin{equation}\label{sig-def}
\sigma(dx)= \Delta(x) \;dx .
\end{equation}
By multiplying $A$ with a suitable scalar constant, we may assume without loss of generality that $\sigma(\Omega)=1$. 

The role of $A$ in this paper is to set a local frame, at each point $x\in\Omega$, relative to which we measure correlations in the array $\xi$. This is particularly relevant in Section \ref{sec2}, where we extend the present discussion to manifolds. The simplest example of $A$ to keep in mind for now is the constant function $A(x)=\vol(\Omega)^{-1/d} I_d$ ($I_d$ is the identity matrix), so that $\sigma(dx)= \vol(\Omega)^{-1} \;dx$ is the uniform probability measure on $\Omega$. 

Given an increasing sequence $M=(M_i)_i$ in $\RR_{>0}$, the {\em pair correlation measure} $\rho_i$ of $\xi$ is defined by
\begin{equation}\label{pc4}
\rho_i f = \frac{M_i}{N_i^2} \sum_{\substack{j_1,j_2=1\\ j_1\neq j_2}}^{N_i} f(M_i^{1/d} A(\xi_{ij_1}) (\xi_{ij_1}-\xi_{ij_2})) ,
\end{equation}
where $f\in\C_c^+(\RR^d)$ (non-negative, continuous with compact support). The sequence $M$ determines the scale on which we measure correlations, and $A(\xi_{ij_1})$ provides a local rescaling of length units near each point $\xi_{ij_1}$, relative to the density of the measure $\sigma$. We call the pair $(A,M)$ a {\em scaling}.

If equidistribution \eqref{dom-ud} is known for some probability measure $\sigma$ with continuous density $\Delta$, then the most canonical choice for $A$ is $A(x)=\Delta^{1/d}(x) I_d$ and $M_i=N_i$, so that \eqref{pc4} captures correlations in units of the average Euclidean distance between the $\xi_{ij}$ near $x$, which is proportional to $(N_i\Delta(x))^{-1/d}$. The point of the present discussion is, however, that we do not assume equidistribution of the array $\xi$, and hence there is no a priori preferred choice of $A$ or $\sigma$.

Note that $\rho_i$ is a locally finite Borel measure on $\RR^d$. It is {\em not} a probability measure. We equip the space of locally finite Borel measures on $\RR^d$ with the vague topology, and say {\em $\xi$ has limiting pair correlation measure $\rho$ for the scaling $(A,M)$}, if $\rho_i$ converges vaguely to $\rho$. That is, if 
\begin{equation}\label{vc}
\lim_{i\to\infty} \rho_i f = \rho f \quad \text{for every $f\in\C_c^+(\RR^d)$.}
\end{equation}
We say $\rho_i$ has a {\em Poisson limit} for the scaling $(A,M)$ if \eqref{vc} holds with $\rho=\vol$. (The constant multiplier in this relation seems arbitrary, but is in fact determined by our scaling of $A$ such that $\sigma(\Omega)=1$.) In this case \eqref{vc} is equivalent to the statement 
\begin{equation}\label{vc222}
\lim_{i\to\infty} \rho_i D = \vol D \quad \text{for every bounded $D\subset\RR^d$ with $\vol\partial D=0$,}
\end{equation}where
\begin{equation}
\rho_i D = \frac{M_i}{N_i^2} \#\{ (j_1,j_2) \in \ZZ_{\neq}^2\cap[1,N_i]^2 \mid \xi_{ij_1}-\xi_{ij_2}\in M_i^{-1/d} A(\xi_{ij_1})^{-1}  D \},
\end{equation}
and $\ZZ_{\neq}^2=\ZZ^2\setminus\{(j,j)\mid j\in\ZZ\}$.

We furthermore say {\em $\rho_i$ has a sub-Poisson limit} if 
\begin{equation}\label{vc001}
\limsup_{i\to\infty} \rho_i f \leq \vol f \quad \text{for every $f\in\C_c^+(\RR^d)$,}
\end{equation}
which again is equivalent to the corresponding statement for bounded $D\subset\RR^d$ with $\vol\partial D=0$.

%\begin{ex}
%*IID in $\Omega$, 
%\end{ex}

In many applications one considers only the pair correlation with respect to the distance between points. We consider here $\dist(x,y)= \| x-y\|$, with $\|\,\cdot\,\|$ the Euclidean norm in $\RR^d$.
The corresponding pair correlation is a locally finite Borel measure on $\RR_{\geq 0}$ defined by
\begin{equation}\label{pc-meas}
\widetilde\rho_i h = \frac{M_i}{N_i^2} \sum_{\substack{j_1,j_2=1\\ j_1\neq j_2}}^{N_i}  h(M_i^{1/d} \|A(\xi_{ij_1})(\xi_{ij_1}-\xi_{ij_2})\|) ,
\end{equation}
for $h\in\C_c^+(\RR_{\geq 0})$. In the spatial statistics literature variants of this are often referred to as Ripley's $K$-function; cf.\ \cite[Sect. 8.3]{Ripley81}.

Define the Borel measure $\omega$ on $\RR_{\geq 0}$ by
\begin{equation}\label{tilderho}
\omega [0,r] = r^d \vol B_1^ d,
\end{equation}
where $B_1^ d$ is the open unit ball. We say {\em $\widetilde\rho_i$ has a Poisson limit} if it converges vaguely to $\omega$.
Note that if $h\in\C_c^+(\RR_{\geq 0})$ then $f\in\C_c^+(\RR^d)$ for $f(x)=h(\|x\|)$. Therefore the vague convergence $\rho_i\to\rho$ implies the vague convergence $\widetilde\rho_i\to\widetilde\rho$ with $\widetilde\rho$ defined by the relation $\widetilde\rho h=\rho f$ with $f(x)=h(\|x\|)$. Thus if $\rho_i$ has a Poisson limit in the vague topology, then so does $\widetilde\rho_i$. 
%That is, for any $r>0$,
%\begin{equation}
%\lim_{i\to\infty}\widetilde\rho_i [0,r] = r^d \vol B_1^ d .
%\end{equation}
We say {\em $\widetilde\rho_i$ has a sub-Poisson limit} if 
\begin{equation}\label{vc002}
\limsup_{i\to\infty} \widetilde\rho_i h \leq \omega h \quad \text{for every $h\in\C_c^+(\RR_{\geq 0})$.}
\end{equation}
The latter statement is equivalent to 
\begin{equation}\label{vc00202}
\limsup_{i\to\infty} \widetilde\rho_i [0,r] \leq r^d \vol B_1^ d \quad \text{for every $r>0$.}
\end{equation}

\begin{thm}\label{dom-thm:torus}
Fix $A$ and $\sigma$ as defined above, and let $\xi$ be a triangular array in $\Omega$. Then the following holds.
\begin{enumerate}[{\rm (i)}]
\item Suppose there is a sequence $M$ with $M_i\to\infty$ and $M_i\leq N_i$, such that $\widetilde\rho_i$ has a sub-Poisson limit for the scaling $(A,M)$. Then $\xi$ is equidistributed in $(\Omega,\sigma)$.
\item Suppose $\xi$ is equidistributed in $(\Omega,\sigma)$. Then there is a sequence $M$ with $M_i\to\infty$ and $M_i\leq N_i$, such that $\rho_i$ has a Poisson limit for the scaling $(A,M)$.
\end{enumerate}
\end{thm}

It is well known that equidistribution does not imply a Poisson pair correlation at the scale $M_i=N_i$. (An elementary example is the triangular array in $[0,1]$ given by $\xi_{ij}=\frac{j}{N_i}$.) Furthermore, a Poisson pair correlation at this scale does not imply that other fine-scale statistics, such as the nearest-neighbour distribution, are Poisson \cite{Baddeley84,ElBaz15,ElBaz15b}. 

The proof of part (i) is split into four lemmas. 
For $x\in\RR^d$, define the counting measure $\hat\mu_i^x$ on $\RR^d$ by
\begin{equation}\label{mu-def}
\hat\mu_i^x f =  \sum_{j=1}^{N_i} f(M_i^{1/d} A(\xi_{ij})(\xi_{ij}-x)) ,
\end{equation}
where $f\in\C_c^+(\RR^d)$. Denote by $\chi_D$ the indicator function of a bounded subset $D\subset\RR^d$. Then
\begin{equation}
\hat\mu_i^x D 
=  \sum_{j=1}^{N_i} \chi_D(M_i^{1/d} A(\xi_{ij})(\xi_{ij}-x)) =   \#\{ j\leq N_i \mid  \xi_{ij} \in  x+M_i^{-1/d} A(\xi_{ij})^{-1}D  \} .
\end{equation}

For $\epsilon>0$, let $\Omega_\epsilon=\Omega+B_\epsilon^d$ be the $\epsilon$-neighbourhood of $\Omega$, where $B_\epsilon^d$ is the open ball of radius $\epsilon$ centered at the origin. The Tietze extension theorem allows us to extend $\Delta$ to a continuous function $\RR^d \to \RR_{>0}$. We also extend $\sigma$ to a locally finite measure outside $\Omega$ via relation \eqref{sig-def}.

It is convenient to work with the following normalised variant of $\hat\mu_i^x$,
\begin{equation}
\mu_i^x = \frac{M_i}{N_i} \hat\mu_i^x.
\end{equation}

\begin{lem}\label{lem0}
Fix a triangular array $\xi$, a sequence $M$ with $M_i\to\infty$ and $M_i\leq N_i$, and a bounded set $D\subset\RR^d$. Then, for $\epsilon>0$,
\begin{equation}\label{dom-eq:conw00}
\lim_{i\to\infty}  \int_{\Omega_\epsilon} \mu_i^x D\; \sigma(dx) 
 = \vol D .
\end{equation}
\end{lem}

\begin{proof}
Since $\Delta$ is uniformly continuous, we have
\begin{equation}\label{dom-eq:conw000}
\begin{split}
\int_{\Omega_\epsilon} \mu_i^x D\; \sigma(dx) 
& =  \frac{M_i}{N_i}  \sum_{j=1}^{N_i} \int_{\Omega_\epsilon} \chi_D(M_i^{1/d} A(\xi_{ij})(\xi_{ij}-x)) \, \Delta(x)\;dx \\
& =  \frac{M_i}{N_i}  \sum_{j=1}^{N_i} \int_{\Omega_\epsilon} \chi_D(M_i^{1/d} A(\xi_{ij})(\xi_{ij}-x)) \, (\Delta(\xi_{ij})+o(1))\;dx \\
& =  \frac{1}{N_i}  \sum_{j=1}^{N_i} \int_{M_i^{1/d}A(\xi_{ij}) (\xi_{ij}-\Omega_\epsilon)} \chi_D(x) \, (1+o(1))\;dx .
\end{split}
\end{equation}
For $M_i$ sufficiently large, we have 
\begin{equation}\label{eighteen}
D\subset M_i^{1/d}A(\xi_{ij}) B_\epsilon^d \subset  M_i^{1/d}A(\xi_{ij}) (\xi_{ij}-\Omega_\epsilon),
\end{equation}
since $\xi_{ij}+B_\epsilon^d\subset \Omega_\epsilon$.
This implies \eqref{dom-eq:conw00}.
\end{proof}

We denote by $\C_c(\Omega^\circ)$ the class of continuous functions $\Omega\to\RR$ with compact support in the interior $\Omega^\circ$ of $\Omega$.

\begin{lem}\label{dom-lem11}
Fix a triangular array $\xi$, a sequence $M$ with $M_i\to\infty$ and $M_i\leq N_i$, and a bounded set $D\subset\RR^d$ with $\vol D>0$. If for every Borel probability measure $\lambda$ on $\Omega$ with density in $\C_c(\Omega^\circ)$ we have
\begin{equation}\label{dom-eq:conw}
\lim_{i\to\infty} \int_{\Omega} \mu_i^x D\; \lambda(dx) = \vol D,
\end{equation}
then $\xi$ is equidistributed. 
\end{lem}

\begin{proof}
Let $f\in\C_c(\Omega^\circ)$ be the density of $\lambda$ with respect to $\sigma$. Then \eqref{dom-eq:conw} states explicitly that
\begin{equation}\label{dom-eq:conw2}
\lim_{i\to\infty} \frac{M_i}{N_i} \sum_{j=1}^{N_i} \int_{\Omega} \chi_D(M_i^{1/d} A(\xi_{ij}) (\xi_{ij}-x)) f(x) \, \sigma(dx)= \vol D \int_{\Omega}  f(x) \, \sigma(dx),
\end{equation}
which by linearity in fact holds for any $f\in\C_c(\Omega^\circ)$, not necessarily probability densities. 
Since $f$ and $A$ are uniformly continuous and $D$ is bounded, we have uniformly for $y\in\Omega$,
\begin{equation}
\int_{\Omega} \chi_D(M_i^{1/d}A(y)(y-x)) f(x) \, \sigma(dx) = f(y) (1+o(1)) \int_{\Omega} \chi_D(M_i^{1/d}A(y)(y-x))  \, \sigma(dx) ,
\end{equation}
and
\begin{equation}
\int_{\Omega} \chi_D(M_i^{1/d}A(y)(y-x))  \, \sigma(dx) = \vol D +o(1),
\end{equation}
uniformly for all $y\in\supp f$. (This follows from the same reasoning as in the proof of Lemma \ref{lem0}, since $\supp f$ avoids an $\epsilon$-neighbourhood of $\partial\Omega$, for some $\epsilon>0$.)
Therefore,
\begin{equation}
\frac{M_i}{N_i} \sum_{j=1}^{N_i} \int_{\Omega} \chi_D(M_i^{1/d} A(\xi_{ij})(\xi_{ij}-x)) f(x) \, \sigma(dx)
=  \frac{\vol D}{N_i} \sum_{j=1}^{N_i} f(\xi_{ij}) +o(1).
\end{equation}
Thus \eqref{dom-eq:conw2} implies for $f\in\C_c(\Omega^\circ)$
\begin{equation}\label{bbb}
\lim_{i\to\infty} \frac{1}{N_i} \sum_{j=1}^{N_i} f(\xi_{ij})  = \int_{\Omega}  f(x) \, \sigma(dx).
\end{equation}
This relation can be extended to $f\in\C_b(\Omega)$ by noting that \eqref{bbb} holds trivially for every constant test function: Any $f\in\C_b(\Omega)$ can be approximated from below by a function in $\C_c(\Omega^\circ)$, and from above by a function in $\C_c(\Omega^\circ)$ plus a constant.
This proves that $\xi$ is equidistributed.
\end{proof}

\begin{lem}\label{dom-lem12}
Fix a triangular array $\xi$ and a bounded set $D\subset\RR^d$. If there is a sequence $M$ with $M_i\to\infty$ and $M_i\leq N_i$ such that 
\begin{equation}\label{dom-eq:conw01}
\lim_{i\to\infty} \int_{\Omega} \big( \mu_i^x D - \vol D \big)^2\sigma(dx) = 0 ,
\end{equation}
then
\begin{equation}\label{dom-eq:conw02}
\lim_{i\to\infty} \int_{\Omega} \mu_i^x D \;\lambda(dx) = \vol D 
\end{equation}
for every Borel probability measure $\lambda$ with square-integrable density (with respect to $\sigma$).
\end{lem}

\begin{proof}
Let $f$ be the density of $\lambda$. By the Cauchy-Schwarz inequality,
\begin{equation}\label{dom-eq:conw0111}
\bigg|\int_{\Omega} \big( \mu_i^x D - \vol D \big) \lambda(dx) \bigg|
 \leq  \bigg(\int_{\Omega} f(x)^2\sigma(dx)\bigg)^{1/2} \bigg(\int_{\Omega} \big( \mu_i^x D - \vol D \big)^2\sigma(dx)\bigg)^{1/2}.
\end{equation}
This converges to zero as $i\to\infty$, which proves \eqref{dom-eq:conw02}.
\end{proof}

\begin{lem}\label{dom-lem13}
Fix a triangular array $\xi$, a sequence $M$ with $M_i\to\infty$ and $M_i\leq N_i$, and a bounded subset $D\subset\RR^d$ with $\vol\partial D=0$. 
Set 
\begin{equation}\label{dom-fDef}
f(x) = \vol\big( (D+x)\cap D \big) .
\end{equation}
Then $f\in\C_c^+(\RR^d)$ and we have, for $\epsilon>0$,
\begin{equation}\label{oh1}
\int_{\Omega_\epsilon} \big( \mu_i^x D - \vol D \big)^2 \sigma(dx) = \rho_i f -(\vol D)^2 (2-\sigma(\Omega_\epsilon))+ \frac{M_i}{N_i} \vol D +o(1).
\end{equation}
\end{lem}

\begin{proof}
By Lemma \ref{lem0}, 
\begin{equation}
\int_{\Omega_\epsilon} \mu_i^x D \,\sigma(dx) = \vol D +o(1),
\end{equation}
and so
\begin{equation}
\int_{\Omega_\epsilon} \big( \mu_i^x D - \vol D \big)^2 \sigma(dx) = \int_{\Omega_\epsilon} \big( \mu_i^x D \big)^2 \sigma(dx) - (\vol D)^2 (2-\sigma(\Omega_\epsilon)) + o(1).
\end{equation}
Furthermore, by the same reasoning as in the proof of Lemma \ref{lem0},
\begin{equation}\label{oh2}
\begin{split}
\int_{\Omega_\epsilon} \big( \mu_i^x D \big)^2 \sigma(dx) & =\frac{M_i^2}{N_i^2} \sum_{j_1,j_2=1}^{N_i}  \int_{\Omega_\epsilon} \chi_D(M_i^{1/d} A(\xi_{ij_1}) (\xi_{ij_1}-x)) \chi_D(M_i^{1/d}A(\xi_{ij_2})(\xi_{ij_2}-x))  \, \sigma(dx) \\
& =\frac{M_i}{N_i^2} \sum_{j_1,j_2=1}^{N_i}  \int_{\RR^d}  \chi_D(x)  \chi_D(x-M_i^{1/d}A(\xi_{ij_1})(\xi_{ij_1}-\xi_{ij_2})) \, dx +o(1).
\end{split}
\end{equation}
The summation over distinct indices $j_1\neq j_2$ yields $\rho_i f$ with $f$ as defined in \eqref{dom-fDef}. The summation over $j_1=j_2$ yields $\frac{M_i}{N_i} \vol D$.

The function $f$ is compactly supported, since $D$ is bounded. To prove continuity, note that for $\|x-y\|<\epsilon$, $|f(x)-f(y)|$ is bounded above by the volume of the $\epsilon$-neighbourhood of $\partial D$. Continuity of $f$ is therefore implied by the assumption $\vol\partial D=0$. 
\end{proof}

\begin{proof}[Proof of Theorem \ref{dom-thm:torus} (i)]
Assume that $\widetilde\rho_i$ has a sub-Poisson limit for some sequence $M$ with $M_i\leq N_i$. 
It follows from \eqref{vc00202} that, for any $\delta>0$, $\widetilde\rho_i$ also has a sub-Poisson limit for the scaling $(A,M^\delta)$ defined by $M_i^\delta = \delta M_i$. To highlight the $\delta$-dependence we write $\widetilde\rho_i=\widetilde\rho_i^\delta$.

Let $D=B_1^d$, and $f$ as defined in \eqref{dom-fDef}. $f(x) = \vol\big( (B_1^d+x)\cap B_1^d \big)$, and note that $f(x)=h(\| x\|)$ where
$h(r) = \vol\big( (B_1^d+r e_0 )\cap B_1^d \big)$
with $e_0$ an arbitrary choice of unit vector. The function $h$ is continuous and compactly supported on $\RR_{\geq 0}$, with $h(0)=\vol B_1^d$.  By assumption $\widetilde\rho_i^\delta$ has a sub-Poisson limit. Hence
\begin{equation}
\limsup_{i\to\infty} \widetilde\rho_i^\delta h \leq \omega h,
\end{equation}
and so
\begin{equation}
\limsup_{i\to\infty} \rho_i^\delta f \leq \int_{\RR^d} f(x) dx = (\vol D)^2.
\end{equation}
With this, Lemma \ref{dom-lem13} shows that, for any $\epsilon,\delta>0$,
\begin{equation}\label{oh123456}
\limsup_{i\to\infty}\int_{\Omega} \big( \mu_i^x D - \vol D \big)^2 \sigma(dx) 
\leq (\sigma(\Omega_\epsilon)-1) (\vol D)^2+ \delta \vol D .
\end{equation}
Since $\vol\partial\Omega=0$ and thus $\sigma(\partial\Omega)=0$, we have $\sigma(\Omega_\epsilon)\to\sigma(\Omega)=1$ as $\epsilon\to 0$. Thus there is a sequence of $\delta_i\to 0$, such that for the scaling $(A,M')$ given by ${M_i}'=\delta_i M_i$ we have
\begin{equation}\label{oh123456789}
\limsup_{i\to\infty}\int_{\Omega} \big( \mu_i^x D - \vol D \big)^2 \sigma(dx) = 0.
\end{equation}
This confirms the hypothesis of Lemma \ref{dom-lem12} for the sequence $M'$. Lemma \ref{dom-lem12} in turn establishes the assumption for Lemma \ref{dom-lem11}, which completes the proof of claim (i).
\end{proof}

\begin{proof}[Proof of Theorem \ref{dom-thm:torus} (ii)]
Since $\xi$ is equidistributed in $(\Omega,\sigma)$ we have, for $\psi\in\C_b(\Omega\times\Omega)$,
\begin{equation}
\lim_{i\to\infty} \frac{1}{N_i^2} \sum_{j_1,j_2=1}^{N_i} \psi(\xi_{ij_1},\xi_{ij_2}) = \int_{\Omega\times\Omega} \psi(x_1,x_2) \,\sigma(dx_1)\,\sigma(dx_2).
\end{equation}
Since $\psi$ is bounded, the above statement remains valid with the diagonal terms $j_1=j_2$ removed. For fixed $M_0>0$ and $f\in\C_c^+(\RR^d)$, apply this asymptotics with the choice $\psi(x_1,x_2)=M_0 f(M_0^{1/d} A(x_1)(x_1-x_2))$, which is bounded continuous. This yields,
\begin{equation}
\lim_{i\to\infty} \frac{M_0}{N_i^2} \sum_{\substack{j_1,j_2=1\\ j_1\neq j_2}}^{N_i} f(M_0^{1/d} A(\xi_{ij_1})(\xi_{ij_1}-\xi_{ij_2})) \\
= M_0 \int_{\Omega\times\Omega}  f(M_0^{1/d} A(x_1)(x_1-x_2)) \,\sigma(dx_1)\,\sigma(dx_2).
\end{equation}
The right hand side can be written as
\begin{multline}
M_0 \int_{\Omega\times\Omega}  f(M_0^{1/d} A(x_1)(x_1-x_2)) \, \Delta(x_1)\,\Delta(x_2) \,dx_1\,dx_2 \\
%& = M_0 \int_{x_1+x_2,x_2\in\Omega}  f(M_0^{1/d} A(x_1+x_2)x_1) \, \Delta(x_1+x_2)\,\Delta(x_2) \,dx_1\,dx_2 \\
 = \int_\Omega \bigg(\int_{M_0^{1/d} (\Omega-x_2)}  f(A(M_0^{-1/d} x_1+x_2)x_1) \, \Delta(M_0^{-1/d} x_1+x_2)\,\Delta(x_2) \,dx_1\bigg) dx_2 .
\end{multline}
Since $f,A$ are continuous and $\Omega$ has boundary of Lebesgue measure zero, this expression converges, as $M_0\to\infty$, to 
\begin{equation}\label{lasteq}
\int_{\Omega}\int_{\RR^d}  f(A(x_2)x_1) \, \Delta(x_2)^2 \,dx_1\,dx_2 = \int_{\Omega}\int_{\RR^d}  f(x_1) \,dx_1\,\sigma(dx_2)=  
\vol f .
\end{equation}
This proves that there is a slowly growing sequence $M_i\to\infty$ such that 
\begin{equation}
\lim_{i\to\infty} \frac{M_i}{N_i^2} \sum_{\substack{j_1,j_2=1\\ j_1\neq j_2}}^{N_i} f(M_i^{1/d} A(\xi_{ij_1})(\xi_{ij_1}-\xi_{ij_2})) = \vol f,
\end{equation}
which proves part (ii) of the theorem.
\end{proof}

\section{Riemannian manifolds}\label{sec2}

Let $(\scrM,g)$ be a compact Riemannian manifold with metric $g$. We denote by $\vol_g$ the corresponding Riemannian volume, and normalise $g$ such that $\vol_g \scrM=1$. The geodesic distance between $x,y\in\scrM$ is denoted $\dist_g(x,y)$. Now consider a triangular array $\xi$ with coefficients in $\scrM$, and define the corresponding pair correlation measure by
\begin{equation}\label{pc-meas2}
\rho_i^{(g)} h = \frac{M_i}{N_i^2} \sum_{\substack{j_1,j_2=1\\ j_1\neq j_2}}^{N_i}  h(M_i^{1/d} \dist_g(\xi_{ij_1},\xi_{ij_2})) .
\end{equation}
In other words, for $r>0$,
\begin{equation}
\rho_i^{(g)} [0,r] 
= \frac{M_i}{N_i^2} \#\{ (j_1,j_2) \in \ZZ_{\neq}^2\cap[1,N_i]^2 \mid  \dist_g(\xi_{ij_1},\xi_{ij_2})\leq M_i^{-1/d} r  \}.
\end{equation}
We say {\em $\rho_i^{(g)}$ has a Poisson limit} for the scaling $M$ if it converges vaguely to $\omega$, with $\omega$ as defined in \eqref{tilderho} (with $\vol$ still the Lebesgue measure in $\RR^d$), and similarly say it has a {\em sub-Poisson limit} if for every $h\in\C_c^+(\RR_{\geq 0})$
\begin{equation}\label{vc003}
\limsup_{i\to\infty} \rho_i^{(g)} h \leq \omega h.
\end{equation}
which is equivalent to the statement
\begin{equation}\label{vc0020273627}
\limsup_{i\to\infty} \rho_i^{(g)} [0,r] \leq r^d \vol B_1^ d \quad \text{for every $r>0$.}
\end{equation}

The following is a corollary of Theorem \ref{dom-thm:torus}.

\begin{thm}
Let $(\scrM,g)$ be a compact Riemannian manifold, and $\xi$ a triangular array with coefficients in $\scrM$.
\begin{enumerate}[{\rm (i)}]
\item Suppose there is a sequence $M$ with $M_i\to\infty$ and $M_i\leq N_i$, such that $\rho_i^{(g)}$ has a sub-Poisson limit for the scaling $M$. Then $\xi$ is equidistributed in $(\scrM,\vol_g)$.
\item Suppose $\xi$ is equidistributed in $(\scrM,\vol_g)$. Then there is a sequence $M$ with $M_i\to\infty$ and $M_i\leq N_i$, such that $\rho_i^{(g)}$ has a Poisson limit for the scaling $M$.
\end{enumerate}
\end{thm}

Part (i) is closely related to, but not implied by, the results in \cite{Steinerberger17} for the choice $\rho_i^{(g)}h$ with $h(r)=\exp(-r^2)$.

\begin{proof}[Proof of (i)]
Consider an atlas $\{(U_\alpha,\varphi_\alpha)\mid \alpha\in \scrA\}$ with $\scrA$ finite. 
We take $\varphi_\alpha(U_\alpha)\subset\RR^d$ to lie in the same copy of $\RR^d$, arranged in such a way that the $\varphi_\alpha(U_\alpha)$ are pairwise disjoint. Now consider a partition of $\scrM$ by the bounded sets $V_\beta$ with $\beta\in\scrB$ and $\scrB$ finite, so that 
\begin{equation}
\bigcup_{\beta\in\scrB} V_\beta = \scrM, \qquad V_\beta\cap V_{\beta'}=\emptyset \; \text{if $\beta\neq\beta'$}, \qquad \vol_g\partial V_\beta=0.
\end{equation}
We assume the partition is sufficiently refined so that for $\beta\in\scrB$ there is a choice of $\alpha(\beta)\in \scrA$ such that $\cl V_\beta\subset U_{\alpha(\beta)}$. We set $\Omega_\beta=\varphi_{\alpha(\beta)} V_\beta$. The disjoint union
\begin{equation}
\Omega = \bigcup_{\beta\in\scrB} \Omega_\beta 
\end{equation}
is a bounded subset of $\RR^d$ with $\vol\partial\Omega=0$. Given a triangular array $\xi$ in $\scrM$ we define a corresponding array $\xi'$, whose $i$th row $(\xi_{ij}')_{j\leq N_i}$ is given by the elements in the set
\begin{equation}
\bigcup_{\beta\in\scrB} \varphi_{\alpha(\beta)}\big( \{ \xi_{ij} \mid j\leq N_i\}\cap V(\beta) \big) .
\end{equation}

By Gram-Schmidt orthonormalisation, there is a continuous function $A:\cl\Omega \to \GL(d,\RR)$ such that the metric $g$ is given at $x\in U_\alpha$ by the positive definite bilinear form 
\begin{equation}
g_x(X,Y) = \langle A(\varphi_{\alpha} x) X , A(\varphi_{\alpha} x) Y \rangle,
\end{equation}
where $\langle \,\cdot\, , \,\cdot\,  \rangle$ is the standard Euclidean inner product. With this choice, and the probability measure $\sigma$ defined as in \eqref{sig-def}, we see that the triangular array $\xi$ is equidistributed in $(\scrM,\vol_g)$ if and only if $\xi'$ is equidistributed in $(\Omega,\sigma)$.

Let us now compare the pair correlation measure $\widetilde\rho_i$ for $\xi'$ in $\Omega$ as defined in \eqref{pc-meas} with $\rho_i^{(g)}$.
For $h\in\C_c^+(\RR_{\geq 0})$, we have 
\begin{equation}
h(M_i^{1/d} \|A(\xi_{ij_1})(\xi_{ij_1}'-\xi_{ij_2}')\|) = 0
\end{equation}
if $\xi_{ij_1}'\in\Omega_\beta$, $\xi_{ij_2}'\in\Omega_{\beta'}$ with $\beta\neq\beta'$ and $M_i$ is sufficiently large. 
This means that the pairs $(j_1,j_2)$ contributing to $\widetilde\rho_i$ form a subset of those contributing to $\rho_i^{(g)}$.
Furthermore, we have
\begin{equation}
\dist_g(x,y) \sim \| A(\varphi_\alpha y) ( \varphi_\alpha x- \varphi_\alpha y) \|
\end{equation}
for $\|\varphi_\alpha x- \varphi_\alpha y\|\to 0$. 
Both facts taken together imply, by the uniform continuity of $h\in\C_c^+(\RR_{\geq 0})$, that
\begin{equation}
\limsup_{i\to\infty} \widetilde\rho_i h \leq \limsup_{i\to\infty} \rho_i^{(g)} h.
\end{equation}
This shows that if $\rho_i^{(g)}$ has a sub-Poisson limit then so does $\widetilde\rho_i$. Theorem \ref{dom-thm:torus} tells us that therefore $\xi'$ is equidistributed in $(\Omega,\sigma)$, and hence (as noted earlier) $\xi$ is equidistributed in $(\scrM,\vol_g)$. This yields claim (i). 
\end{proof}

\begin{proof}[Proof of (ii).]
Since $\xi$ is equidistributed in $(\scrM,\vol_g)$ we have, for $\psi\in\C(\scrM\times\scrM)$,
\begin{equation}
\lim_{i\to\infty} \frac{1}{N_i^2} \sum_{j_1,j_2=1}^{N_i} \psi(\xi_{ij_1},\xi_{ij_2}) = \int_{\scrM\times\scrM} \psi(x_1,x_2) \, 
\vol_g(dx_1)\,\vol_g(dx_2).
\end{equation}
Since $\psi$ is bounded, the above statement remains valid with the diagonal terms $j_1=j_2$ removed. For fixed $M_0>0$ and $h\in\C_c^+(\RR_{\geq 0})$, apply this asymptotics with the choice $\psi(x_1,x_2)=M_0 h(M_0^{1/d} \dist_g(x_1,x_2))$, which is bounded continuous. This yields,
\begin{equation}
\lim_{i\to\infty} \frac{M_0}{N_i^2} \sum_{\substack{j_1,j_2=1\\ j_1\neq j_2}}^{N_i} h(M_0^{1/d} \dist_g(\xi_{ij_1},\xi_{ij_2})) \\
= M_0 \int_{\scrM\times\scrM}  h(M_0^{1/d} \dist(x_1,x_2)) \,\vol_g(dx_1)\,\vol_g(dx_2).
\end{equation}
The limit $M_0\to 0$ can be calculated in local charts, which leads to the same calculation as in the proof of Theorem \ref{dom-thm:torus} (ii). 
\end{proof}

\begin{appendix}

\section{Flat tori}

It is instructive to adapt the discussion in Section \ref{sec1} to the case of a multidimensional torus $\TT$. This provides an alternative approach to the results in \cite{Steinerberger19}. We represent the torus as $\TT=\RR^d/\scrL$, with $\scrL\subset\RR^d$ a Euclidean lattice of unit covolume (for example the integer lattice $\scrL=\ZZ^d$). The required modifications are as follows. 

\begin{enumerate}[A.]
\item
Replace $\Omega$ by $\TT$ throughout Section \ref{sec1}, and note that $\C(\TT)=\C_b(\TT)=\C_c(\TT)$. 
\item
The coefficients of the triangular array are written as $\xi_{ij}+\scrL\in\TT$ with $\xi_{ij}\in\RR^d$. 
\item
Set for simplicity $A(x)=I_d$, so that $\sigma=\vol$ is the uniform probability measure on $\TT$. (It is of course possible to adapt the argument also for general continuous $A:\TT\to\GL(d,\RR)$.)
\item
The definition of the pair correlation measure $\rho_i$ in \eqref{pc4} is replaced by
\begin{equation}\label{pc4-TT}
\rho_i f = \frac{M_i}{N_i^2} \sum_{\substack{j_1,j_2=1\\ j_1\neq j_2}}^{N_i} \sum_{m\in\scrL} f(M_i^{1/d} (\xi_{ij_1}-\xi_{ij_2}+m)) ,
\end{equation}
and \eqref{pc-meas} by
\begin{equation}\label{pc-meas-TT22}
\widetilde\rho_i h = \frac{M_i}{N_i^2} \sum_{\substack{j_1,j_2=1\\ j_1\neq j_2}}^{N_i} \sum_{m\in\scrL}   h(M_i^{1/d} \|\xi_{ij_1}-\xi_{ij_2}+m\|) .
\end{equation}
(Note that the pair correlation measure \eqref{pc-meas2} for the Riemannian distance on $\TT$,
\begin{equation}
\dist_g(x,y)=\min_{m\in\scrL}\{ \| x-y+m \| \} ,
\end{equation}
satisfies the relation $\widetilde\rho_i h=\rho_i^{(g)} h$, for $h\in\C_c^+(\RR_{\geq 0})$ and $M_i$ sufficiently large.)
\item
The discussion around \eqref{dom-eq:conw00} is replaced by the following.
For $x\in\TT$, define the measure $\mu_i^x$ on $\RR^d$ by
\begin{equation}
\mu_i^x f = \frac{M_i}{N_i}  \sum_{j=1}^{N_i} \sum_{m\in\scrL} f(M_i^{1/d} (\xi_{ij}-x+m)) ,
\end{equation}
where $f\in\C_c^+(\RR^d)$. That is, for a bounded subset $D\subset\RR^d$, we have
\begin{equation}
\begin{split}
\mu_i^x D 
& = \frac{M_i}{N_i}  \sum_{j=1}^{N_i} \sum_{m\in\scrL} \chi_D(M_i^{1/d} (\xi_{ij}-x+m)) \\
& = \frac{M_i}{N_i}  \#\{ j\leq N_i \mid  \xi_{ij} \in  x+M_i^{-1/d} D +\scrL \} ;
\end{split}
\end{equation}
the second equality holds of $M_i$ is sufficiently large so that $M_i^{-1/d} D$ does not intersect any translate $M_i^{-1/d} D+m$ with $m\in\scrL\setminus\{0\}$.
Then
\begin{equation}\label{dom-eq:conw00-TT}
\int_\TT \mu_i^x D\; dx 
=  \frac{M_i}{N_i}  \sum_{j=1}^{N_i} \sum_{m\in\scrL} \int_{\TT} \chi_D(M_i^{1/d} (\xi_{ij}-x+m)) \, dx 
= \vol D .
\end{equation}
\item
In the statement of Lemma \ref{dom-lem13} no $\epsilon$ is needed, and \eqref{oh1} is replaced by the identity
\begin{equation}\label{oh1-TT}
\int_{\TT} \big( \mu_i^x D - \vol D \big)^2 dx = \rho_i f -(\vol D)^2 + \frac{M_i}{N_i} \vol D ,
\end{equation}
which follows from the following calculation, replacing \eqref{oh2},
\begin{equation}
\begin{split}
\int_{\TT} \big( \mu_i^x D \big)^2 dx & =\frac{M_i^2}{N_i^2} \sum_{j_1,j_2=1}^{N_i} \sum_{m_1,m_2\in\scrL} \int_{\TT} \chi_D(M_i^{1/d}(\xi_{ij_1}-x+m_1)) \chi_D(M_i^{1/d}(\xi_{ij_2}-x+m_2))  \, dx \\
& =\frac{M_i}{N_i^2} \sum_{j_1,j_2=1}^{N_i} \sum_{m\in\scrL} \int_{\RR^d}  \chi_D(x)  \chi_D(x-M_i^{1/d}(\xi_{ij_1}-\xi_{ij_2}+m)) \, dx .
\end{split}
\end{equation}
\item
For the proof of the second part of the theorem, we use instead
$$\psi(x_1,x_2)=M_0 \sum_{m\in\scrL} f(M_0^{1/d} (x_1-x_2+m)),$$ which is continuous on $\TT\times\TT$, with $f\in\C_c^+(\RR^d)$ as before. The assumed equidistribution implies
\begin{equation}
\lim_{i\to\infty} \frac{M_0}{N_i^2} \sum_{\substack{j_1,j_2=1\\ j_1\neq j_2}}^{N_i} \sum_{m\in\scrL}  f(M_0^{1/d} (\xi_{ij_1}-\xi_{ij_2}+m)) 
= M_0 \sum_{m\in\scrL}  \int_{\TT\times\TT}  f(M_0^{1/d} (x_1-x_2+m)) \,dx_1\,dx_2
\end{equation}
which evaluates to $\vol f$.
\end{enumerate}
\end{appendix}


\begin{thebibliography}{99}


\bibitem{Aistleitner18}
C. Aistleitner, T. Lachmann and F. Pausinger, Pair correlations and equidistribution, J. Number Th. 182 (2018), 206--220.

\bibitem{Aistleitner18c}
C. Aistleitner, T. Lachmann and N. Technau,
There is no Khintchine threshold for metric pair correlations, arXiv:1802.02659

\bibitem{Aistleitner18b}
C. Aistleitner, G. Larcher and M. Lewko, Additive Energy and the Hausdorff dimension of the exceptional set in metric pair correlation problems, Israel J. Math. 222 (2017), 463--485.

\bibitem{Baddeley84}
A.J. Baddeley and B.W. Silverman,
A cautionary example on the use of second-order methods for analyzing point patterns, Biometrics 40 (1984) 1089--1093.

\bibitem{Bloom18}
T.F. Bloom, S. Chow, A. Gafni and A. Walker, Additive energy and the metric Poissonian property, 
Mathematika 64 (2018), 679--700.


\bibitem{ElBaz15}
D. El-Baz, J. Marklof and I. Vinogradov, The two-point correlation function of the fractional parts of $\sqrt{n}$ is Poisson, Proc. AMS 143 (2015), 2815--2828.

\bibitem{ElBaz15b}
D. El-Baz, J. Marklof and I. Vinogradov, The distribution of directions in an affine lattice: two-point correlations and mixed moments, IMRN (2015), 1371--1400.

\bibitem{Grepstad17}
S. Grepstad and G. Larcher, On pair correlation and discrepancy, Arch. Math. 109
(2017), 143--149.

\bibitem{HeathBrown10}
D. R. Heath-Brown, Pair correlation for fractional parts of $ \alpha n^2 $, Math. Proc. Cambridge Philos. Soc. 148 (2010),
385--407.

\bibitem{Hinrichs18}
A. Hinrichs, L. Kaltenb\"ock, G. Larcher, W. Stockinger and M. Ullrich, On a multi-dimensional
Poissonian pair correlation concept and uniform distribution, arXiv:1809.05672

\bibitem{Larcher18}
G. Larcher and W. Stockinger, Some negative results related to Poissonian pair correlation problems, arXiv:1803.052361

\bibitem{Marklof00}
J. Marklof, The $n$-point correlations between values of a linear form, with an appendix by Z. Rudnick, Erg. Th. Dyn. Sys. 20 (2000), 1127--1172.

\bibitem{Marklof03}
J. Marklof and A. Str\"ombergsson, Equidistribution of Kronecker sequences along closed horocycles, Geom. Funct. Anal. 13 (2003), 1239--1280.

\bibitem{Marklof18}
J. Marklof and N. Yesha, Pair correlation for quadratic polynomials mod 1, Compositio Math. 154 (2018), 960--983.

\bibitem{Ripley81}
B.D. Ripley, {\em Spatial Statistics}, Wiley-Interscience, 1981. 

\bibitem{Rudnick98}
Z. Rudnick and P. Sarnak, The pair correlation function of fractional parts of polynomials, Comm. Math. Phys. 194 (1998), 61--70.

\bibitem{Steinerberger17}
S. Steinerberger, Localized quantitative criteria for equidistribution, Acta Arith. 180 (2017), 183--199. 

\bibitem{Steinerberger18}
S. Steinerberger, Poissonian pair correlation and discrepancy, Indagationes Math. 29 (2018), 1167--
1178 

\bibitem{Steinerberger19}
S. Steinerberger, Poissonian pair correlation in higher dimensions, arXiv:1812.10458 

\bibitem{Walker18}
A. Walker, The primes are not metric Poissonian, Mathematika 64 (2018), 230--236.

%\bibitem{Weyl16}
%H.\,Weyl, \"Uber die Gleichverteilung von Zahlen mod.~Eins,
%Math. Ann. 77 (1916), 313--352.

\end{thebibliography}
\end{document}